\documentclass[3p]{elsarticle}
\pdfoutput=1

\usepackage{amssymb}
\usepackage{amsfonts}
\usepackage{amsmath}
\usepackage{amsthm}
\usepackage{booktabs}
\usepackage{caption}
\usepackage{enumitem}
\usepackage{soul}
\usepackage{thm-restate}
\usepackage{xcolor}
\usepackage{array}
\usepackage{makecell}
\usepackage[hidelinks]{hyperref}
\usepackage{breakcites}
\usepackage{endnotes}
\usepackage{longtable}
\usepackage{caption}
\usepackage{geometry}
\usepackage{setspace}
\usepackage{natbib}
\bibpunct[, ]{(}{)}{,}{a}{}{,}%
\def\newblock{\ }%

\usepackage{makecell}

\biboptions{authoryear}

\makeatletter
\def\ps@pprintTitle{%
  \let\@oddhead\@empty
  \let\@evenhead\@empty
  \let\@oddfoot\@empty
  \let\@evenfoot\@oddfoot
}
\makeatother

\newtheorem{theorem}{Theorem}

\newcommand\given[1][]{\:#1\vert\:}

\renewcommand{\hl}{}

\begin{document}

\begin{frontmatter}

%\title{Elsevier \LaTeX\ template\tnoteref{mytitlenote}}
%\tnotetext[mytitlenote]{Fully documented templates are available in the elsarticle package on \href{http://www.ctan.org/tex-archive/macros/latex/contrib/elsarticle}{CTAN}.}
%\title{Recent Advances in Vehicle Routing with Stochastic Demands: Elementary Branch-Price-and-Cut, Preventive Restocking and Correlated Demands}
\title{Recent Advances in Vehicle Routing with Stochastic Demands: Bayesian Learning for Correlated Demands and Elementary Branch-Price-and-Cut}

%% Group authors per affiliation:
%\author{Elsevier\fnref{myfootnote}}
%\address{Radarweg 29, Amsterdam}
%\fntext[myfootnote]{Since 1880.}

%% or include affiliations in footnotes:
\author[poly]{Alexandre M. Florio\corref{mycorrespondingauthor}}
\cortext[mycorrespondingauthor]{Corresponding author}
\ead{aflorio@gmail.com}
\author[poly]{Michel Gendreau}
\ead{michel.gendreau@polymtl.ca}
\author[uniwien]{Richard F. Hartl}
\ead{richard.hartl@univie.ac.at}
\author[tum]{Stefan Minner}
\ead{stefan.minner@tum.de}
\author[poly,scale,puc]{Thibaut Vidal}
\ead{thibaut.vidal@polymtl.ca}

\address[poly]{CIRRELT \& Department of Mathematics and Industrial Engineering, Polytechnique Montreal, Quebec, Canada}
\address[uniwien]{Department of Business Decisions and Analytics, University of Vienna, Austria}
\address[tum]{Technical University of Munich, School of Management, and Munich Data Science Institute, Germany}
\address[scale]{SCALE-AI Chair in Data-Driven Supply Chains}
\address[puc]{Departamento de Inform\'{a}tica, Pontif\'{i}cia Universidade Cat\'{o}lica do Rio de Janeiro (PUC-Rio), Brazil}

\begin{abstract}
We consider the vehicle routing problem with stochastic demands (VRPSD), a stochastic variant of the well-known VRP in which demands are only revealed upon arrival of the vehicle at each customer. Motivated by the significant recent progress on VRPSD research, we begin this paper by summarizing the key new results and methods for solving the problem. In doing so, we discuss the main challenges associated with solving the VRPSD under the chance-constraint and the restocking-based perspectives. Once we cover the current state-of-the-art, we introduce two major methodological contributions. First, we present a branch-price-and-cut (BP\&C) algorithm for the VRPSD under optimal restocking. The method, which is based on the pricing of elementary routes, compares favorably with previous algorithms and allows the solution of several open benchmark instances. Second, we develop a demand model for dealing with correlated customer demands. The central concept in this model is the ``external factor'', which represents unknown covariates that affect all demands. We present a Bayesian-based, iterated learning procedure to refine our knowledge about the external factor as customer demands are revealed. This updated knowledge is then used to prescribe optimal replenishment decisions under demand correlation. Computational results demonstrate the efficiency of the new BP\&C method and show that cost savings above 10\% may be achieved when restocking decisions take account of demand correlation. Lastly, we motivate a few research perspectives that, as we believe, should shape future research on the VRPSD.
\end{abstract}


\end{frontmatter}

%\linenumbers

\section{Introduction}
The vehicle routing problem \citep[VRP;][]{toth2014vehicle} is a foundational problem in transportation logistics. In this paper, we consider the VRP with stochastic demands (VRPSD), a VRP variant in which customer demands are only revealed upon arrival of the vehicle at each customer. The VRPSD has been extensively studied, with the first contribution dating back to more than 50 years ago \citep{tillman1969multiple}. The high interest on the VRPSD is explained by the applicability and difficulty of the problem, which combines elements from discrete and stochastic optimization. From a practical standpoint, the VRPSD arises in applications such as fuel delivery, waste collection, logistics of dairy products, courier services, sludge disposal, emergency logistics, distribution of industrial gases and replenishment of vending machines \citep[see, e.g.,][]{larson1988transporting,chan2001multiple,singer2002fleet,ChepuriHomemDeMello2005,yan2013planning}.

Our first goal in this paper is to review and analyze recent contributions, summarize key results and provide new research avenues on the VRPSD. This in-depth review is motivated by substantial recent progress on VRPSD research. Key contributions include new exact algorithms, theoretical results on restocking policies, and new recourse policies, \hl{which describe rules (and their associated costs) to handle the possibility that the total demand along a route may exceed the vehicle capacity.} These advances enabled the solution of larger instances under more realistic demand distributions and recourse policies, which were known research gaps in this area \citep{Gendreauetal2016}. At the same time, the recent progress shed light on important but relatively less studied VRPSD variants, and suggested future research directions.

Following the review of the state-of-art, the focus of this paper becomes methodological. We first present a new branch-price-and-cut (BP\&C) algorithm for the VRPSD under optimal restocking. This new BP\&C method is based on the pricing of elementary routes (hence, \emph{elementary} BP\&C), and improves considerably upon the dominance-based BP\&C by \cite{Florio_2020}. In particular, the new BP\&C algorithm finds the optimal solution of several previously unsolved benchmark instances, and reduces drastically the runtime requirements for solving instances also solved by the dominance-based BP\&C.

Our second methodological contribution is a Bayesian demand model and an optimal restocking policy for the case of correlated customer demands. The central concept in this demand model is the \emph{external factor}, which is a ``state of the world'' variable that represents unknown demand predictors. During route execution, each time the demand of a customer is revealed we update our beliefs about the external factor in a Bayesian fashion, and use the updated knowledge to prescribe optimal replenishment decisions. We also provide a theoretical convergence result about the Bayesian learning procedure.

Before we proceed, we clarify that the scope of this paper is the \emph{static} VRPSD. In the static paradigm, all customers are known at routing planning stage. In the alternative \emph{dynamic} paradigm, new customers may appear while routes are already in execution. For a review on dynamic VRPs, we refer to the recent work by \cite{Soeffker_2021}.

In summary, our main contributions are the following:
\begin{enumerate}[label=(\roman*)]
\item We review the state-of-art on the VRPSD with a focus on the new contributions and results since the review by \cite{Gendreauetal2016}. We cover both the restocking-based and the chance-constraint modeling paradigms, or \emph{perspectives}, and discuss the most challenging aspects when solving the VPRSD under either perspective. We also motivate future research on the VRPSD by discussing a few important research gaps.
\item We introduce a new state-of-art BP\&C algorithm for the VRPSD under optimal restocking, which is based on the pricing of elementary routes and strong completion bounds. This algorithm compares very favorably with the dominance-based BP\&C by \cite{Florio_2020}. In addition to closing several previously unsolved benchmark instances, we also provide, for the first time, optimal solutions to difficult instances where the total expected demand along a route is allowed to exceed the vehicle capacity.
\item We propose a demand model for correlated demands, together with a Bayesian procedure to update, as customer demands are revealed, the knowledge of a state of the world variable that represents unknown demand covariates. Further, we extend the optimal restocking policy by \cite{Yee_1980} to prescribe optimal replenishment decisions under demand correlation.
\end{enumerate}

The remainder of this paper is organized as follows: Section \ref{sec:probdef} formulates the VRPSD under two modeling paradigms, or perspectives: the restocking-based and the chance-constraint perspectives. Section \ref{sec:review} reviews recent key contributions on the VRPSD and discusses the challenging aspects related to each perspective. Section \ref{sec:bpnc} introduces the elementary BP\&C algorithm for the VRPSD under optimal restocking. Section \ref{sec:corr} presents the Bayesian-inspired demand model and restocking policy for dealing with correlated customer demands. Section \ref{sec:results} reports computational results on the elementary BP\&C algorithm and the optimal restocking policy for correlated demands. Section \ref{sec:conclusions} concludes the paper by summarizing the key takeaways and motivating future research directions.

\section{The VRPSD: One Problem, Two Perspectives}\label{sec:probdef}
In contrast to its deterministic counterpart, in the VRPSD the total demand along a route may exceed the vehicle capacity. Two broad model classes arise, depending on how capacity violations are managed. In the VRPSD with restocking, vehicles may perform replenishment trips to the depot to increase the capacity available along routes. In the VRPSD without restocking, capacity constraints are enforced in a probabilistic way to reduce the likelihood of capacity violations. In this section, we first introduce common notation, and then discuss both modeling perspectives.

\subsection{General Definitions}
The VRPSD is defined on a complete digraph $\mathcal{G}=(\mathcal{V},\mathcal{A})$, where $\mathcal{V}=\{0,\ldots,n\}$ is the set of nodes and $\mathcal{A}=\{(i,j):i,j\in\mathcal{V}\}$ is the set of arcs. Node 0 identifies the depot and the node set $\mathcal{V}^{+}=\mathcal{V}\setminus\{0\}$ identifies customers. A discrete random demand $\xi_{i}$ is associated with each customer $i\in\mathcal{V}^{+}$, and a deterministic cost $c_{ij}$ is associated with each arc $(i,j)\in\mathcal{A}$. A fleet of $m$ homogeneous vehicles is available and each vehicle has a capacity of $Q$.

We define a route as a non-empty sequence of customers $\theta$, and we let $\Theta$ be the set of all feasible routes. We also let $g:\Theta\mapsto\mathbb{R}_{+}$ be a route cost function. The specifications of $\Theta$ and $g(\cdot)$ depend on the model class, as further discussed in Sections \ref{sec:pswithrestock} and \ref{sec:psnorestock}. With binary decision variables $z_{\theta}$, $\theta\in\Theta$, to indicate whether route $\theta$ belongs to the solution, the VRPSD can be modeled generically as a set-partitioning problem:
\begin{align}
&\text{min}&&\sum_{\theta\in\Theta}g(\theta)z_{\theta},\label{eq:spobj}\\
&\text{s.t.}&&\sum_{\theta\in\Theta}\mathbb{I}(i\in\theta)z_{\theta}=1,&i\in\mathcal{V}^{+},\label{eq:spptt}\\
&&&\sum_{\theta\in\Theta}z_{\theta}\leq m,\label{eq:spmaxv}\\
&&&z_{\theta}\in\{0,1\},&\theta\in\Theta,\label{eq:spdom}
\end{align}
where $\mathbb{I}(\cdot)$ is the indicator function.

The model minimizes the total cost of the routes selected in the solution, under the constraints that each customer must be served by exactly one vehicle (constraints \eqref{eq:spptt}) and the fleet size limit is respected (constraint \eqref{eq:spmaxv}). Next, we show how this generic set-partitioning model is specialized to the VRPSD with and without restocking.

\subsection{VRPSD with Restocking}\label{sec:pswithrestock}
In the restocking-based VRPSD all demand must be served, and vehicles are allowed to replenish at the depot to increase the capacity available along routes. Therefore, in this model class we have the notion of a restocking policy, which is a set of rules that govern, during route execution, when vehicles replenish at the depot.

The cost of a route is the sum of the costs of the arcs traversed along the route; therefore, it depends on the restocking policy adopted. Many restocking policies have been proposed, for example, optimal restocking \citep{Yee_1980}, detour-to-depot \citep{Dror_1989}, rule-based restocking \cite{Salavati_Khoshghalb_2019rule}, and the switch policy \citep{FlorioReopt2021}. \hl{The route cost under each of those policies is efficiently computed by dynamic programming. In Section }\ref{sec:optrestock}\hl{, we detail the algorithm to compute costs when the route is executed under optimal restocking.}

To establish the feasible region of the restocking-based VRPSD, we first define the set $\Theta_{\Omega}$ of all possible customer sequences:
\begin{equation}\nonumber
\Theta_{\Omega}=\left\{(v_{1},\ldots,v_{H}):H\geq1\,\wedge\,v_{i}\in\mathcal{V}^{+}\,\wedge\,v_{i}\neq v_{j}\text{ if }i\neq j\right\}.
\end{equation}
Then, the set of feasible routes is given by:
\begin{equation}\nonumber
\Theta=\left\{\theta\in\Theta_{\Omega}:\sum_{i\in\theta}\mathbb{E}\bigl[\xi_{i}\bigr]\leq fQ\right\},
\end{equation}
where $f\geq1$ is the so-called \emph{load factor} parameter.

Traditionally, the VRPSD has been studied assuming a load factor $f=1$, under the argument that, on average, vehicles should be able to serve all demand along a route without restocking. However, in some applications \citep[e.g., waste collection, see][]{Jaunich_2016} restocking (or unloading) is the norm rather than the exception. Therefore, sensible models and algorithms for the VRPSD should allow for larger load factors, even though, as we shall see in Section \ref{sec:resbpnc}, problem difficulty increases considerably with the load factor.

As many efficient algorithms for the VRPSD with restocking are based on the integer L-shaped method \citep{Laporte_1993}, we also present the formulation upon which those algorithms are based. Given binary variables $y_{ij}$, $(i,j)\in\mathcal{A}$, to identify the arcs used in a solution, the two-index model is as follows:
\begin{align}
&\text{min}&&\sum_{(i,j)\in\mathcal{A}}c_{ij}y_{ij}+\mathbb{E}\bigl[\mathcal{Q}^{\pi}(\mathbf{y})\bigr],\label{eq:twoindexobj}\\
&\text{s.t.}&&\sum_{j\in\mathcal{V}^{+}}y_{0j}\leq m,&\nonumber\\
&&&\sum_{i\in\mathcal{V}}y_{ij}=1,&j\in\mathcal{V}^{+},\nonumber\\
&&&\sum_{j\in\mathcal{V}}y_{ij}=1,&i\in\mathcal{V}^{+},\nonumber\\
&&&\sum_{i\in S}\sum_{j\notin S}y_{ij}\geq\left\lceil\sum_{i\in S}\mathbb{E}\bigl[\xi_{i}\bigr]/fQ\right\rceil,&S\subset\mathcal{V}^{+},\label{eq:twoindexrcc}\\
&&&y_{ij}\in\{0,1\},&(i,j)\in\mathcal{A},\nonumber
\end{align}
where $\mathbf{y}=[y_{ij}]_{(i,j)\in\mathcal{A}}$.

The model resembles two-index deterministic VRP models with capacity cut constraints \cite[see, e.g.,][]{toth2014vehicle}, which, in the VRPSD, are adjusted to take account of the load factor. The term $\mathbb{E}[\mathcal{Q}^{\pi}(\mathbf{y})]$ in the objective function \eqref{eq:twoindexobj} refers to the expected restocking cost given the set of planned routes induced by $\mathbf{y}$ and assuming that restocking policy $\pi$ is in effect. Hence, the two-index formulation can be interpreted as a two-stage stochastic program, where the first-stage decision is the set of planned routes. However, the VRPSD with restocking is actually a multi-stage stochastic problem since a decision (e.g., to replenish or not) is taken after each customer is served.

\subsection{Chance-constrained VRPSD}\label{sec:psnorestock}
In certain applications, restocking may not be practical. For example, the depot may be located far from the demand zones and customers may not strictly require all demand to be satisfied. This gives rise to VRPSD models where the set of feasible routes is defined by the following probabilistic constraint (or chance-constraint):
\begin{equation}\nonumber
\Theta=\left\{\theta\in\Theta_{\Omega}:\mathbb{P}\left[\sum_{i\in\theta}\xi_{i}\leq Q\right]\geq1-\epsilon\right\},
\end{equation}
where parameter $\epsilon$ regulates the maximum allowed probability of exceeding the vehicle capacity.

In chance-constrained models, the cost of a route $\theta=(v_{1},\ldots,v_{H})$ is deterministic and given by $g(\theta)=c_{0v_{1}}+\sum_{i\in\{2,\ldots,H\}}c_{v_{i-1}{i}}+c_{v_{H}0}$. Hence, in this perspective the demand uncertainty affects only the feasible solution space.

\section{Review of Recent Contributions}\label{sec:review}
Rather than cataloging all VRPSD work since \cite{tillman1969multiple}, we review in detail the key theoretical and methodological contributions that appeared since the review by \cite{Gendreauetal2016}. For a more encyclopedic survey of the VRPSD, and stochastic routing in general, we refer to \cite{Oyola_2016part2,Oyola_2016part1}. Our review is split into two broad sections, VRPSD with and without restocking, which correspond to the two modeling perspectives discussed in Section \ref{sec:probdef}. \hl{For convenience, Tables }\ref{tab:reviewR}\hl{ and }\ref{tab:reviewCC}\hl{ highlight the main contributions of each work reviewed in this section.}

\begin{table}[t]
\centering
\begingroup
\small
\renewcommand{\arraystretch}{0.75}
\caption{\label{tab:reviewR}\hl{Key Recent VRPSD Contributions (Restocking Perspective)}}
\begin{tabular}{ll}
\toprule
Reference & Main contributions\\
\midrule
\cite{Bertazzi_2018} & \makecell[l]{Theoretical bounds on the cost of routes under optimal\\restocking and detour-to-depot policies.}\\
\cite{Florio_2020} & \makecell[l]{BP\&C algorithm for the VRPSD under optimal restocking;\\results on instances with high value of stochastic solution.}\\
\cite{Florio_2020ejor} & \makecell[l]{Unified mixed-integer linear model for the single-VRPSD\\under optimal restocking.}\\
\cite{Bertazzi_2020} & \makecell[l]{Forward dynamic programming procedure to approximate\\the cost of a route under optimal restocking.}\\
\cite{Louveaux_2018} & \makecell[l]{First exact algorithm (integer L-shaped) for the VRPSD\\under optimal restocking; results on i.i.d. customer demands.}\\
\cite{Salavati_Khoshghalb_2019exact} & \makecell[l]{Integer L-shaped method for the VRPSD under optimal\\restocking; results on non-i.i.d. customer demands.}\\
\cite{Salavati_Khoshghalb_2019hybrid} & \makecell[l]{Preventive restocking policy that balances risk and expected\\cost of failures; integer L-shaped algorithm.}\\
\cite{Salavati_Khoshghalb_2019rule} & \makecell[l]{Customer-specific, rule-based restocking policy; integer\\L-shaped algorithm.}\\
\cite{FlorioReopt2021} & \makecell[l]{First exact algorithm (BP\&C) for a VRPSD under a\\(partial) reoptimization recourse policy.}\\
\cite{Florio_2021} & \makecell[l]{B\&P algorithm for a duration-constrained VRPSD with\\restocking; Monte Carlo method for checking route feasibility.}\\
\cite{delavega2021} & \makecell[l]{Integer L-shaped algorithm for a VRPSD with time windows\\under rule-based restocking.}\\
\cite{ymro2022} & \makecell[l]{New families of valid inequalities and improved L-shaped\\algorithm for the VRPSD under optimal restocking.}\\
\textbf{This paper} & \makecell[l]{Elementary BP\&C algorithm; optimal restocking policy for\\positively correlated customer demands.}\\
\bottomrule
\end{tabular}
\endgroup
\end{table}

\subsection{VRPSD with Restocking}
A well-known result on the VRPSD is that, given a customer sequence, the optimal restocking policy is of a threshold type and can be computed by dynamic programming \citep{Yee_1980}.  There are both theoretical and empirical results that justify adopting optimal restocking in place of detour-to-depot, which is the policy employed by all exact methods proposed before the review by \cite{Gendreauetal2016}. \cite{Bertazzi_2018} show that the expected cost of a route executed under detour-to-depot is up to 100\% higher than the expected cost of the same route executed under optimal restocking. They also show, by means of a somewhat extreme example, that the expected cost of the best detour-to-depot route (that is, the route with minimum expected cost when executed under detour-to-depot) to visit a given set of customers may be 50\% higher than the expected cost of the best optimal restocking route (that is, the route with minimum expected cost when executed under optimal restocking) to visit the same set of customers.

\cite{Florio_2020ejor} provide empirical evidence that the best optimal restocking route outperforms the best detour-to-depot route by a moderate margin, also in less pathological instances than those used by \cite{Bertazzi_2018} to demonstrate theoretical results. The difference in expected cost between the two policies, when routes are optimized for each policy, may exceed 8\%. One key insight from \cite{Florio_2020ejor} is that differences in expected costs between both policies are minor when the total expected demand along a route does not exceed the vehicle capacity. This empirical result is confirmed by \cite{Florio_2020}, where best VRPSD solutions under detour-to-depot and optimal restocking are compared on several benchmark instances with a load factor $f=1$. This comparison shows that there is little benefit in adopting optimal restocking when the total demand along a route cannot exceed the vehicle capacity.

\cite{Florio_2020ejor} studies the single-vehicle version of the VRPSD and propose a mixed-integer linear model for finding the optimal route under optimal restocking. The single-VRPSD, which is a special case of the VRPSD when $\sum_{i\in\mathcal{V}^{+}}\mathbb{E}[\xi_{i}]\leq fQ$, finds practical applications, for example, in waste collection operations where customers are partitioned in districts \citep{ghiani2014}. The single-VRPSD is an impressively hard problem: currently, there is no algorithm able to solve general instances with as few as 25 customers under optimal restocking. \cite{Bertazzi_2020} propose a forward dynamic programming procedure to approximate efficiently the cost of a route under optimal restocking. Even though this cost can be computed exactly by solving the original dynamic program by \cite{Yee_1980}, a fast approximation is useful within rollout procedures to find good-quality single-VRPSD solutions, considering, again, the absence of effective exact methods for solving this variant.

The first algorithm for the VRPSD under optimal restocking was proposed by \cite{Louveaux_2018}. This integer L-shaped method introduces lower bounds on the restocking cost of any feasible solution. Two methods for computing these so-called \emph{global} bounds are proposed, the most effective of which relies on independent and identically distributed (i.i.d.) customer demands. Furthermore, \cite{Louveaux_2018} adapt the lower bounding functions originally proposed by \cite{jabali2014} for the case of detour-to-depot restocking. These functions compute lower bounds on the restocking cost of fractional solutions consisting of partial routes. Again, special bounds are developed for the case of i.i.d. demands. Overall, the method is effective for solving instances with up to 100 customers and 3 vehicles, under i.i.d. demands with up to nine possible demand values.

\cite{Salavati_Khoshghalb_2019exact} propose another integer L-shaped algorithm for the VRPSD with optimal restocking. The methodology developments are similar to \cite{Louveaux_2018} in that global lower bounds and adaptations of the bounding functions from \cite{jabali2014} are derived. Differently from \cite{Louveaux_2018}, however, results on instances with non-identical demand distributions are also reported. Demands follow discrete distributions with up to five possible demand values, and the largest solved instance has 60 customers and 4 vehicles. \hl{Still on integer L-shaped methods, }\cite{ymro2022}\hl{ propose an improved algorithm based on stronger partial route inequalities and three families of the so-called \emph{split} inequalities, which decompose the recourse cost by route. Compared to previous methods, this algorithm requires significant less runtime for solving the instances proposed by }\cite{Louveaux_2018}\hl{, and is also able to solve some of those instances for the first time.}

The first BP\&C algorithm for the VRPSD under optimal restocking was introduced by \cite{Florio_2020}. Two main ideas enabled the application of the branch-and-price framework for solving the problem under general independent demand distributions. First, a pricing algorithm based on a backward labeling strategy, in which the intermediate results of the optimal restocking dynamic program (i.e., the cost-to-go functions) are stored in each label. Second, the derivation of label dominance rules for the stochastic case. \cite{Florio_2020} also determine empirically that the value of stochastic solution (VSS) of many instances proposed and solved by \cite{Louveaux_2018} are very small (i.e., less than 1\%), which is mostly not the case in instances with fewer customers per route (i.e., eight or less) and Poisson distributed demands.

Exact methods for solving the VRPSD under alternative restocking policies have also been proposed. \cite{Salavati_Khoshghalb_2019rule} consider rule-based policies, which are policies that prescribe preventive replenishment trips whenever the remaining vehicle capacity falls below preset customer-specific thresholds. \cite{Salavati_Khoshghalb_2019hybrid} propose a restocking policy that considers both the risk of a failure (i.e., not having enough capacity to serve the demand of a customer) and the expected cost of failures along the untravelled portion of the route; hence, a \emph{hybrid} policy. The main idea in rule-based and hybrid policies is to benefit from preventive replenishment without having to solve a dynamic program to compute the optimal restocking policy. Integer L-shaped algorithms are developed to solve the VRPSD under both policies, and the results show that best rule-based or hybrid restocking solutions are only marginally worse than the best optimal restocking solutions.

Note that the VRPSD models from \cite{Louveaux_2018}, \cite{Salavati_Khoshghalb_2019exact}, \cite{Salavati_Khoshghalb_2019rule} and \cite{Salavati_Khoshghalb_2019hybrid} assume a load factor $f=1$. The model used by \cite{Florio_2020} allows larger load factors, but only instances with $f=1$ are solved to optimality. As mentioned, compared to the detour-to-depot policy, the relative cost savings achieved by more advanced restocking policies are usually only marginal when the load factor is set to $f=1$.

One possible step towards better solutions to the restocking-based VRPSD could be to solve the problem under policies that allow not only preventive replenishment, but also \emph{sequencing} decisions. In this direction, \cite{FlorioReopt2021} propose a BP\&C algorithm to solve the VRPSD under the switch policy, which is a recourse policy that allows swapping the visiting order of any two adjacent customers along a planned route. However, even though problem difficulty increases considerably under the switch policy, the best solutions under this policy are only marginally better than the best optimal restocking solutions, even in instances with a load factor above one.

Because of the possibility of replenishment trips, in the restocking-based VRPSD the arrival time at each customer is, in general, stochastic. Therefore, variants of the problem with timing constraints (e.g., time windows or deadlines) become more challenging, as these constraints need to be treated either as probabilistic or soft constraints (i.e., have their violations penalized). \cite{Florio_2021} propose a branch-and-price algorithm to solve a VRPSD under optimal restocking in which routes must only comply with a maximum duration chance-constraint. The resulting problem is significantly more challenging than the usual setting, as the absence of a load limit along a feasible route (given by $fQ$ in the usual model) precludes the separation of rounded capacity cuts and the use of completion bounds based on the load limit. Finally, \cite{delavega2021} develop an integer L-shaped algorithm to solve a VRPSD variant with time windows under a rule-based restocking policy, where an additional recourse action is taken in case of time window violations.

\begin{table}[t]
\centering
\begingroup
\small
\renewcommand{\arraystretch}{0.75}
\caption{\label{tab:reviewCC}\hl{Key Recent VRPSD Contributions (Chance-constraint / Robust Perspective)}}
\begin{tabular}{lcl}
\toprule
Reference & Main contributions\\
\midrule
\cite{dinh2018exact} & \makecell[l]{Branch-and-cut and BP\&C methods for a chance-constrained\\VRPSD under scenario-based demand distributions.}\\
\cite{Noorizadegan_2018} & \makecell[l]{B\&P algorithms for the chance-constrained VRPSD under\\independent demands.}\\
\cite{sluijk2021} & \makecell[l]{Feasibility bounds to determine route compliance with chance-\\constraints; multi-label pricing strategy for a two-echelon setting.}\\
\cite{Pessoa_2021} & \makecell[l]{Reformulation of the robust VRPSD with knapsack uncertainty as a\\deterministic heterogeneous fleet capacitated VRP.}\\
\cite{Ghosal_2020} & \makecell[l]{Efficient capacity cuts separation and branch-and-cut algorithms\\for the distributionally robust VRPSD.}\\
\cite{munari2019} & \makecell[l]{BP\&C algorithm for a robust VRP with both travel time and\\demand uncertainty; compact model to obtain robust solutions.}\\
\cite{Gounaris_2016} & \makecell[l]{Characterization of demand uncertainty sets for which it is possible\\to efficiently determine route feasibility.}\\
\bottomrule
\end{tabular}
\endgroup
\end{table}

\subsection{Chance-constrained VRPSD}
In spite of relatively less attention devoted to this perspective in recent years, significant progress has been achieved in the chance-constrained VRPSD. As seen in Section \ref{sec:psnorestock}, in this variant the demand stochasticity affects only the feasible solution space. Given a subset $S\subset\mathcal{V}^{+}$, the key challenge in the chance-constrained VRPSD is to compute efficiently, and without strong distributional assumptions, the minimum number of vehicles required to fully serve all customers in $S$ at a service level of $1-\epsilon$. We denote this quantity by $r_{\epsilon}(S)$:
\begin{equation}\nonumber
r_{\epsilon}(S)=\min\left\{k\in\mathbb{N}:\mathbb{P}\left[\sum_{i\in S}\xi_{i}\leq kQ\right]\geq1-\epsilon\right\}.
\end{equation}

\cite{dinh2018exact} propose lower bounds on $r_{\epsilon}(S)$ and extend the edge-based chance-constrained VRPSD formulation by \cite{Laporte_1989} to the case of general (i.e., possibly dependent) customer demands. The proposed bounds are computationally tractable whenever the quantiles of the distribution of $\sum_{i\in S}\xi_{i}$ can be readily computed, which is the case, for example, in multivariate normal and scenario-based demand distributions. The edge-based formulation serves as base for branch-and-cut algorithms that are capable of solving instances with up to 45 customers. Further, \cite{dinh2018exact} present a BP\&C algorithm, in which relaxed chance-constrained feasible q-routes are used to guarantee tractability of the pricing problem. With respect to capacity constraints, this relaxation defines an extended route set $\Theta'\supset\Theta$:
\begin{equation}\nonumber
\Theta'=\left\{\theta\in\Theta_{\Omega}:\sum_{i\in\theta}\mathbb{E}\bigl[\xi_{i}\bigr]\leq Q'\right\},
\end{equation}
where $Q'$ is computed by solving a chance-constrained knapsack problem in a preprocessing step. Overall, the BP\&C method is effective for solving instances with up to 55 customers.

\cite{Noorizadegan_2018} propose a branch-and-price algorithm for the chance-constrained VRPSD when demands are independent and follow Poisson, Gaussian or scenario-based distributions. An elementary pricing strategy is adopted, and dominance rules are derived for each demand distribution.

\cite{sluijk2021} study a two-echelon VRPSD where second-echelon routes must comply with probabilistic capacity constraints. Lower bounds are computed by column generation, and the methods employed to verify route feasibility are also applicable to the single-echelon case. In particular, feasibility bounds are derived to compute lower and upper limits $\underline{Q}$ and $\overline{Q}$ such that a route $\theta$ is feasible if $\sum_{i\in\theta}\mathbb{E}[\xi_{i}]\leq\underline{Q}$, and infeasible if $\sum_{i\in\theta}\mathbb{E}[\xi_{i}]\geq\overline{Q}$. Similar bounds based on the variance of $\sum_{i\in\theta}\xi_{i}$ are proposed. When route feasibility cannot be determined with those bounds, alternative strategies based on convoluting demand distributions (in case of independent demands) and Monte Carlo simulation (in case of general demands) are employed.

Most work on chance-constrained VRPSDs assume that demand distributions are known. This assumption is reasonable, considering that the VRP is an operational problem that is solved for many periods, and that scenario-based (or empirical) demand distributions can be retrieved from historical data. In cases where such data might not be available, however, the VRPSD may be modeled as a robust or distributionally robust optimization problem. \cite{Pessoa_2021} reformulates the robust VRPSD with knapsack uncertainty as a heterogeneous fleet capacitated VRP; therefore, allowing the direct solution of the robust version by algorithms developed for the deterministic variant. \cite{Ghosal_2020} study the distributionally robust VRPSD, where the demand distribution is only partially known. The proposed branch-and-cut method relies on an efficient separation of capacity cuts, which is possible when demands belong to ambiguity sets that satisfy a subadditivity property. \cite{munari2019} investigate a robust VRP with both travel time and demand uncertainty. In addition to a specialized BP\&C algorithm, they also propose a compact model that can be used to obtain robust solutions in instances with up to 100 customers. \cite{Gounaris_2016} characterize demand uncertainty sets for which it is possible to efficiently determine route feasibility. We note that optimal solutions to robust models tend to assume worst-case demand distributions and, consequently, lead to underutilization of vehicle capacity and additional fleet requirements \citep{dinh2018exact}.

\section{Branch-Price-and-Cut for the VRPSD under Optimal Restocking}\label{sec:bpnc}
Branch-and-price is an effective method for solving several VRP variants. We refer to \cite{Costa_2019} for an in-depth review of the methodology. When solving VRPs by branch-and-price, the column generation subproblem is normally a variant of the elementary resource-constrained shortest path problem \citep[RCSP;][]{Feillet_2004}. By and large, efficient algorithms for deterministic VRPs relax the elementary condition and solve the pricing problem with a labeling procedure that relies on dominance rules to discard unpromising partial paths. Experience has shown, however, that some stochastic VRPs are efficiently solved by an elementary pricing strategy, in which the combinatorial growth of labels is controlled only with completion bounds. In this section, we present an elementary BP\&C algorithm for the VRPSD under optimal restocking that improves considerably upon the dominance-based BP\&C by \cite{Florio_2020}.

\subsection{Route Cost under Optimal Restocking}\label{sec:optrestock}
The elementary BP\&C algorithm is based on the generic set-partitioning model \eqref{eq:spobj}-\eqref{eq:spdom}. Before describing the algorithm, we specify the cost function $g(\theta)$ when route $\theta$ is executed under optimal restocking. As in \cite{FlorioReopt2021}, we use the functionals $\phi'(i,j,q,\Phi(\cdot))$ and $\phi''(i,j,\Phi(\cdot))$ to facilitate the presentation:
\begin{gather*}
\phi'(i,j,q,\Phi(\cdot))=c_{ij}+\mathbb{E}\Bigl[\Upsilon_{\xi_{j}}^{q}(c_{j0}+c_{0j})+\Phi(q+Q\Upsilon_{\xi_{j}}^{q}-\xi_{j})\Bigr],\\
\phi''(i,j,\Phi(\cdot))=c_{i0}+c_{0j}+\mathbb{E}\Bigl[\Upsilon_{\xi_{j}}^{Q}(c_{j0}+c_{0j})+\Phi(Q+Q\Upsilon_{\xi_{j}}^{Q}-\xi_{j})\Bigr].
\end{gather*}

Functional $\phi'(i,j,q,\Phi(\cdot))$ calculates the expected remaining cost once customer $i$ is fully served, given that customer $j$ is visited directly after $i$, the remaining vehicle capacity is $q$, and the cost-to-go once $j$ is fully served, as a function of the remaining capacity, is given by $\Phi:\{0,\ldots,Q\}\mapsto\mathbb{R}$. The quantity $\Upsilon_{\xi}^{q}=\max\{0,\lceil(\xi-q)/Q\rceil\}$ refers to the number of replenishment trips required to fully serve a demand of $\xi$ when the remaining capacity is $q$. Similarly, $\phi''(i,j,\Phi(\cdot))$ calculates the expected remaining cost once customer $i$ is fully served, given that the vehicle replenishes at the depot before visiting customer $j$.

Given a route $\theta=(v_{1},\ldots,v_{H})$, $\Phi_{\theta}^{k}(q)$ denotes the expected cost-to-go along $\theta$ once customer $v_{k}$, $k\in\{1,\ldots,H\}$, is fully served and the remaining vehicle capacity is $q$. This cost is computed by solving the dynamic program:
\begin{equation}\nonumber
\Phi_{\theta}^{k}(q)=\begin{cases}
\min\bigl\{\phi'(v_{k},v_{k+1},q,\Phi_{\theta}^{k+1}(\cdot)),\phi''(v_{k},v_{k+1},\Phi_{\theta}^{k+1}(\cdot))\bigr\},& \text{if }k<H,\\
c_{v_{H}0},& \text{if }k=H.
\end{cases}
\end{equation}

Hence, the cost of route $\theta$ under optimal restocking, denoted by $g^{*}(\theta)$, is given by:
\begin{equation}\nonumber
g^{*}(\theta)=\phi'(0,v_{1},Q,\Phi_{\theta}^{1}(\cdot)).
\end{equation}

\subsection{Valid Inequalities}
To improve the continuous relaxation bound given by the set-partitioning model, we separate rounded capacity and subset row inequalities. Rounded capacity cuts (RCCs) correspond to constraint \eqref{eq:twoindexrcc} of the two-index formulation. When translated to the set-partitioning model, these cuts are of the following form:
\begin{equation}\label{eq:rccs}
\sum_{\theta\in\Theta}\sum_{i\in S}\sum_{j\notin S}a_{\theta}^{ij}z_{\theta}\geq\left\lceil\sum_{i\in S}\mathbb{E}\bigl[\xi_{i}\bigr]/fQ\right\rceil,\qquad\forall S\subset\mathcal{V}^{+},
\end{equation}
where $a_{\theta}^{ij}$ is a binary coefficient that indicates whether route $\theta$ traverses arc $(i,j)$.

In addition to RCCs, we also separate subset row cuts \citep[SRCs;][]{Jepsen_2008} defined over customer triplets:
\begin{equation}\label{eq:srcs}
\sum_{\theta\in\Theta}\left\lfloor\frac{1}{2}\sum_{i\in S}\mathbb{I}(i\in\theta)\right\rfloor z_{\theta}\leq1,\qquad\forall S\subset\mathcal{V}^{+},|S|=3.
\end{equation}

SRCs play a key role in the efficiency of the elementary BP\&C algorithm. In dominance-based BP\&C, the separation of these cuts must be done in a controlled way, since each cut induces a new resource over which dominance must also be verified \citep{Poggi_2014}. Since the elementary approach does not rely on dominance rules, SRCs can be separated in a more aggressive way.

\subsection{Pricing Problem}
We refer to the linear relaxation of formulation \eqref{eq:spobj}-\eqref{eq:spdom} with a restricted set of routes $\Theta'\subset\Theta_{\Omega}$ as the restricted master problem (RMP). We let $\mathcal{C}$ and $\mathcal{J}$ be sets of customer sets for which RCCs and SRCs are separated and added to the RMP, respectively. Further, given a solution to the RMP, we let $\alpha_{i}$ ($i\in\mathcal{V}^{+}$), $\beta$, $\gamma_{S}$ ($S\in\mathcal{C}$) and $\delta_{S}$ ($S\in\mathcal{J}$) be the dual values associated with constraints \eqref{eq:spptt}, \eqref{eq:spmaxv}, \eqref{eq:rccs} and \eqref{eq:srcs}, respectively. When solving the relaxed set-partitioning model by column generation, the pricing problem identifies a route with a negative reduced cost, that is, a route $\theta\in\Theta_{\Omega}\setminus\Theta'$ such that:
\begin{equation}\nonumber
g^{*}(\theta)-\zeta(\theta)<0,
\end{equation}
where $\zeta(\theta)$ is the contribution of the dual values to decrease the reduced cost of $\theta$:
\begin{equation}\label{eq:duals}
\zeta(\theta)=\sum_{i\in\theta}\alpha_{i}+\beta+\sum_{S\in\mathcal{C}}\sum_{i\in S}\sum_{j\notin S}a_{\theta}^{ij}\gamma_{S}+\sum_{S\in\mathcal{J}}\left\lfloor\frac{1}{2}\sum_{i\in S}\mathbb{I}(i\in\theta)\right\rfloor\delta_{S}.
\end{equation}

\subsection{Backward Labeling}
We solve the pricing problem with a backward labeling procedure, in which a label represents a partial path from a customer up to the depot. The choice for backward (instead of forward) labeling is justified by noting that the computation of the route cost $g^{*}(\theta)$ starts at the last customer along route $\theta$. Therefore, performing labeling backwards allows us to store, in each label, the intermediate results of the dynamic program, that is, the cost-to-go value function. These costs are then readily available for all possible backward extensions of a partial path.

Given a label $\ell$, $\ell_{\theta}\in\Theta_{\Omega}$ denotes the partial path represented by $\ell$; $\ell_{v}\in\mathcal{V}^{+}$ indicates the current customer; $\ell_{q}\in\{0,\ldots,fQ\}$ denotes the available load; $\ell_{\zeta}\in\mathbb{R}$ represents the total contribution of the dual values to the reduced cost; and $\ell_{\Phi(\cdot)}:\{0,\ldots,Q\}\mapsto\mathbb{R}$ represents the cost-to-go once $\ell_{v}$ is fully served, as function of the remaining vehicle capacity. Label attributes are initialized as described in Table \ref{tab:attribs}. Given a label $\ell^{2}$ obtained by extending a label $\ell^{1}$ to a customer $i$, the cost-to-go function $\ell^{2}_{\Phi(\cdot)}$ is computed by performing one step of the optimal restocking dynamic programming algorithm for each $q\in\{0,\ldots,Q\}$:
\begin{equation}\label{eq:ctgupdt}
\ell^{2}_{\Phi(q)}=\min\bigl\{\phi'(i,\ell^{1}_{v},q,\ell^{1}_{\Phi(\cdot)}),\phi''(i,\ell^{1}_{v},\ell^{1}_{\Phi(\cdot)})\bigr\},
\end{equation}
and the remaining attributes of $\ell^{2}$ are updated as described in Table \ref{tab:attribs}. \hl{Essentially, the partial path of $\ell^{2}$ is set by prepending customer $i$ to the partial path of $\ell^{1}$. The current customer of $\ell^{2}$ is set to~$i$. The available load decreases by the mean demand of customer $i$. Lastly, the total contribution of dual values along the partial path of $\ell^{2}$ is computed by Equation }\eqref{eq:duals}\hl{.}

\begin{table}[h!]
\centering
\begingroup
\small
\renewcommand{\arraystretch}{0.825}
\caption{\label{tab:attribs}Label Attributes, Initialization and Updating Rules}
\begin{tabular}{llrr}
\toprule
Attribute & Description & Initialization$^a$ & Update$^b$\\
\midrule
$\ell_{\theta}$ & Partial path & $(i)$ & $\ell^{2}_{\theta}=(i)\oplus\ell^{1}_{\theta}$\\
$\ell_{v}$ & Current customer & $i$ & $\ell^{2}_{v}=i$\\
$\ell_{q}$ & Available load & $fQ-\mathbb{E}\bigl[\xi_{i}\bigr]$ & $\ell^{2}_{q}=\ell^{1}_{q}-\mathbb{E}\bigl[\xi_{i}\bigr]$\\
$\ell_{\Phi(q)}\quad q\in\{0,\ldots,Q\}$ & Cost-to-go function & $c_{i0},\quad\forall q$ & See Equation \eqref{eq:ctgupdt}\\
$\ell_{\zeta}$ & Sum of dual values & $\zeta((i))$ & $\zeta((i)\oplus\ell^{1}_{\theta})$\\
\bottomrule
\end{tabular}
\caption*{$^a$Assuming initialization to customer $i$.\\$^b$Assuming label $\ell^{2}$ is created by extending label $\ell^{1}$ to customer $i$; $\oplus$ denotes sequence concatenation.}
\endgroup
\end{table}

The labeling procedure initially creates labels corresponding to the backward extensions from the depot to each customer. Each label $\ell$ is extended to each customer $i\notin\ell_{\theta}$ such that $\mathbb{E}[\xi_{i}]\leq\ell_{q}$. Routes with a negative reduced cost are saved for inclusion in the RMP, and the combinatorial growth of labels is controlled with completion bounds, as detailed next.

\subsection{Completion Bounds}
A completion bound is a lower bound on the reduced cost of all routes that can be generated from a label. Labels with nonnegative bounds can be discarded because they cannot generate negative reduced cost routes. Although completion bounds may also be employed in dominance-based BP\&C, when pricing only elementary routes it is possible to derive stronger bounds that leverage the elementary condition. In this section, we present strengthened versions of the bounds employed by \cite{Florio_2020}.

In what follows, let $\Theta_{\ell}\subset\Theta_{\Omega}$ be the set of routes that can be generated from a label $\ell$, and let $\sigma_{ij}=\alpha_{j}+\sum_{S\in\mathcal{C}}\mathbb{I}(i\in S,j\notin S)\gamma_{S}$ be the reduced cost decrease due to constraints \eqref{eq:spptt} and \eqref{eq:rccs} when extending a label from customer $j$ to customer $i$.

\subsubsection{Knapsack Bound}

The reduced cost of the route represented by a label $\ell$ is given by $g^{*}(\ell_{\theta})-\ell_{\zeta}$. Note that $g^{*}(\ell_{\theta'})\geq g^{*}(\ell_{\theta})$ for all $\theta'\in\Theta_{\ell}$. Considering that only the dual values associated with constraints \eqref{eq:spptt} and \eqref{eq:rccs} may promote reduced cost decrease (the dual values associated with inequalities \eqref{eq:srcs} are nonpositive), the reduced costs of all routes in $\Theta_{\ell}$ are lower-bounded by:
\begin{equation}\nonumber
g^{*}(\ell_{\theta})-\ell_{\zeta}-\textsc{kp}(\ell),
\end{equation}
where \textsc{kp}$(\ell)$ is the optimal solution value of a $\{0,1\}$-knapsack problem with capacity $\ell_{q}$, and with the value and weight of each item $j\in\mathcal{V}^{+}\setminus\ell_{\theta}$ given by $\max_{(i,j)\in\mathcal{A}}\sigma_{ij}$ and $\mathbb{E}[\xi_{j}]$, respectively. This bound improves upon the unbounded knapsack bound proposed by \cite{Florio_2020}, since the elementary condition is used to forbid items corresponding to customers already visited.

\subsubsection{RCSP Bound}\label{sec:rcspbound}
Given a customer set $S$, we define the auxiliary graph $\mathcal{G}'(S)=(\mathcal{V},\mathcal{A}'(S))$, with set of nodes $\mathcal{V}$ (as in our original graph $\mathcal{G}$) and set of arcs $\mathcal{A}'(S)=\{(i,j)\in\mathcal{A}:i\notin S,j\neq0\}$. With each arc $(i,j)\in\mathcal{A}'(S)$ a modified cost $\overline{c}_{ij}=c_{ij}-\sigma_{ij}$ is associated. We further denote by \textsc{rcsp}$(S,q,i)$ the RCSP in graph $\mathcal{G}'(S)$ from the depot to customer $i$, where the initial capacity at the depot is $q$ and each visit to customer $j$ consumes a capacity of $\mathbb{E}[\xi_{j}]$.

Given a label $\ell$, $\ell_{\Phi(Q)}$ determines a lower bound on the cost-to-go once customer $\ell_{v}$ is fully served. Therefore, for any customer set $\mathcal{M}\subset\mathcal{V}^{+}$ a completion bound on $\ell$ is given by:
\begin{equation}\label{eq:rcspcb}
\textsc{rcsp}\!\left(\ell_{\theta}\cap\mathcal{M},\ell_{q}+\mathbb{E}\bigl[\xi_{\ell_{v}}\bigr],\ell_{v}\right)+\alpha_{\ell_{v}}+\ell_{\Phi(Q)}-\ell_{\zeta}.
\end{equation}

The bound given by \eqref{eq:rcspcb} improves upon the RCSP-based bound by \cite{Florio_2020} in that the shortest path from the depot to node $\ell_{v}$ cannot pass through a customer $i\in\ell_{\theta}\cap\mathcal{M}$. As in \cite{Florio_2021evrp}, at each iteration of the pricing algorithm we define the set $\mathcal{M}$ to contain customers that promote large reduced cost decrease per unit of resource consumed (e.g., the $M$ customers with largest $\alpha_{i}/\mathbb{E}[\xi_{i}]$ ratios, where $M$ is a parameter). The idea is to make bounds tighter when customers in $\mathcal{M}$ are already visited by a label. Before the labeling procedure begins, \textsc{rcsp}$(S,q,i)$ is pre-computed for all $S\subset\mathcal{M}$, $q\in\{0,\ldots,fQ\}$ and $i\in\mathcal{V}^{+}$. Then, each time a label is generated, bound \eqref{eq:rcspcb} can be evaluated in constant time.

\subsection{Cut Separation and Branching}
Lower bounds are obtained by alternating between column generation and separation of RCCs until there are no negative reduced cost columns nor violated RCCs left. Then, SRCs are separated and added to the RMP whenever the maximum possible lower bound increase due to SRCs exceeds 0.0075\% (at the root node of the branch-and-bound tree) or 0.03\% (at other branch-and-bound nodes). This lower bound increase is computed by considering only the routes already added to the RMP. Column generation and separation of RCCs start again once SRCs are separated, and the procedure repeats until there are no more SRCs that promote a lower bound increase above the required threshold. RCCs are separated with the CVRPSEP package by \cite{Lysgaard_2004}, and SRCs are separated by enumeration. Finally, up to eight SRCs are added to the RMP per iteration to promote faster convergence and to avoid solving the pricing problem to optimality too often.

We adopt the same branching rule as in \cite{Florio_2020}. Whenever the linear bound obtained at a branch-and-bound node corresponds to a fractional solution, we branch on an arc $(i,j)$ such that $\sum_{\theta\in\Theta'}a_{\theta}^{ij}z_{\theta}$ is fractional. Typically, many such arcs are available in a fractional solution, so we select the arc that, upon branching, leads to the highest lower bound increase when considering only the columns and rows already generated.

\section{A Demand Model and Restocking Policy for Correlated Demands}\label{sec:corr}
We now propose a demand model based on Bayesian inference to handle the case of correlated customer demands. Furthermore, we extend the optimal restocking algorithm, described in Section~\ref{sec:optrestock} for the case of independent demands, to prescribe optimal replenishment decisions under demand correlation.

\subsection{Demand Model}\label{sec:baydmodel}
The main concept in our demand model is the external factor, which represents censored demand covariates (or predictors) such as, for example, weather conditions, short-term trends and occurrence of events. We assume that the effect of the external factor is uniform across all customer demands. A ``high'' external factor increases the likelihood of observing above-average demands, while a ``low'' factor causes the opposite effect. Therefore, our model represents the case where all customer demands are \emph{positively} correlated. The external factor is unknown at both route planning and route execution stages. The key idea is to start with some prior knowledge about the external factor, and update this knowledge in a Bayesian fashion as customer demands are revealed.

Throughout this section, we assume a fixed route $\theta=(v_{1},\ldots,v_{H})$. For ease of notation, we reindex the customers in $\mathcal{V}^{+}$ in such a way that $v_{i}=i$ for all $v_{i}\in\theta$. We let $\overline{\chi}\in\mathbb{R}_{+}$ be the true (but unknown) external factor, and we let $\hat{\chi}_{i}$ be continuous random variables (r.vs.) representing our best knowledge about $\overline{\chi}$ before each customer $i\in\theta$ is visited. We let $\mu_{i}$, $i\in\theta$, be customer demand rates, as obtained, e.g., by taking the average of historical demands. To model positive correlation, we assume that the external factor affects demands in a multiplicative way, that is, $\mathbb{E}[\xi_{i}|\overline{\chi}]=\mu_{i}\overline{\chi}$, and define r.vs. $\lambda_{i}\equiv\mu_{i}\hat{\chi}_{i}$ to represent our best knowledge about demand rates under the effect of $\overline{\chi}$. We further denote by $f_{\tilde{c}}(\cdot)$ the probability density function of a continuous r.v. $\tilde{c}$, and by $p_{\tilde{d}}(\cdot)$ the probability mass function of a discrete r.v. $\tilde{d}$.

From Bayes' theorem, we have for all $i\in\theta$:
\begin{equation} \label{eq:bayes}
f_{\lambda_{i}|\xi_{i}}(l|x^{(i)})=\frac{f_{\lambda_{i}}(l)\cdot p_{\xi_{i}|\lambda_{i}}(x^{(i)}|l)}{p_{\xi_{i}}(x^{(i)})}.
\end{equation}

In Bayesian terms, the left-hand side of \eqref{eq:bayes} is the \emph{posterior} distribution of $\lambda_{i}$; the first term in the numerator corresponds to the \emph{prior} distribution of $\lambda_{i}$, and the second term is the likelihood function. The term in the denominator is the normalization constant, so the resulting posterior is a valid probability density function. Hence, we condition on the realized demand of customer $i$ to improve our understanding about $\lambda_{i}$. After customer $i$ is visited and a demand of $x^{(i)}$ is observed, the updated knowledge about the external factor is represented by the r.v. $\hat{\chi}_{i+1}\equiv(\lambda_{i}|\xi_{i}=x^{(i)})/\mu_{i}$.

The proposed learning mechanism can be applied irrespective of customer demands and external factor distributions. However, in order to derive closed-form expressions, it is necessary to place distributional assumptions. In what follows, we assume that $\xi_{i}$ is Poisson distributed with rate $\lambda_{i}$. Further, r.v. $\hat{\chi}_{1}$, which represents the prior knowledge about the external factor, is gamma distributed with shape and scale parameters given by $k_{0}$ and $s_{0}$, respectively. These distributions are particularly convenient because the gamma distribution is the conjugate prior of the Poisson distribution \citep{hoff2009first}; therefore, r.vs. $\hat{\chi}_{2},\ldots,\hat{\chi}_{H}$ are also gamma distributed. Given these assumptions, we obtain the following distributions before and after the first customer is visited:
\begin{align}
\lambda_{1}\equiv\mu_{1}\hat{\chi}_{1}&\sim\Gamma\left(k_{0},\mu_{1}s_{0}\right),\nonumber\\
\lambda_{1}\given\xi_{1}=x^{(1)}&\sim\Gamma\left(k_{0}+x^{(1)},\frac{\mu_{1}s_{0}}{1+\mu_{1}s_{0}}\right),\nonumber\\
\hat{\chi}_{2}\equiv\frac{(\lambda_{1}\given\xi_{1}=x^{(1)})}{\mu_{1}}&\sim\Gamma\left(k_{0}+x^{(1)},\frac{s_{0}}{1+\mu_{1}s_{0}}\right),\nonumber
\end{align}
where $\Gamma(k,s)$ denotes the gamma distribution with shape and scale parameters given by $k$ and $s$, respectively.

With the updated knowledge about the external factor after the first customer is visited, which is represented by r.v. $\hat{\chi}_{2}$, we determine $\lambda_{2}\equiv\mu_{2}\hat{\chi}_{2}$ and the learning procedure iterates. In general, after observing demands $x^{(1)},\ldots,x^{(i-1)}$ we have:
\begin{equation}\nonumber
\hat{\chi}_{i}\sim\Gamma\left(k_{0}+X_{i-1},\frac{s_{0}}{1+\sum_{j=1}^{i-1}\mu_{j}s_{0}}\right),
\end{equation}
where
\begin{equation}\nonumber
X_{i}=\sum_{j=1}^{i}x^{(j)}.
\end{equation}

The following theorem shows that the learning procedure takes the external factor knowledge asymptotically closer to the true external factor:
\begin{theorem}\label{thm:converg}
Consider a sequence $\theta=(1,2,\ldots)$ with infinitely many customers. Then, the sequence of r.vs. $(\hat{\chi}_{i})_{i\in\theta}$ converges in mean square to $\overline{\chi}$.
\end{theorem}
\begin{proof}
$(\hat{\chi}_{i})_{i\in\theta}$ converges in mean square to $\overline{\chi}$ if and only if:
\begin{equation} \label{eq:convergencemq}
0=\lim_{i\to\infty}\mathbb{E}\bigl[(\hat{\chi}_{i}-\overline{\chi})^{2}\bigr]=\lim_{i\to\infty}\left(\mathrm{Var}(\hat{\chi}_{i})+\mathbb{E}\bigl[\hat{\chi}_{i}-\overline{\chi}\bigr]^{2}\right).
\end{equation}

We show that both terms on the right-hand side of \eqref{eq:convergencemq} converge to zero, and hence the equality holds. First, note that the variance of the gamma distributed r.v. $\hat{\chi}_{i}$ is given by:
\begin{equation}\label{eq:varextf}
\mathrm{Var}(\hat{\chi}_{i})=\left(k_{0}+X_{i-1}\right)\left(\frac{s_{0}}{1+\sum_{j=1}^{i-1}\mu_{j}s_{0}}\right)^{2}.
\end{equation}

The numerator in \eqref{eq:varextf} grows linearly with $i$, and the denominator grows quadratically with $i$. The other terms are constant, so it follows that $\mathrm{Var}(\hat{\chi}_{i})=0$ when $i\to\infty$.

We now show that $\mathbb{E}[\hat{\chi}_{i}]$ converges to $\overline{\chi}$. By conditioning on $\xi_{1},\ldots,\xi_{i-1}$, and considering again that $\hat{\chi}_{i}$ is gamma distributed, we obtain:
\begin{equation}\label{eq:condexp}
\mathbb{E}\bigl[\hat{\chi}_{i}\given\xi_{1},\ldots,\xi_{i-1}\bigr]=\left(k_{0}+\sum_{j=1}^{i-1}\xi_{j}\right)\left(\frac{s_{0}}{1+\sum_{j=1}^{i-1}\mu_{j}s_{0}}\right).
\end{equation}

Applying the law of iterated expectations and reorganizing the terms of \eqref{eq:condexp}, we have:
\begin{equation}\label{eq:explaw}
\mathbb{E}\bigl[\hat{\chi}_{i}\bigr]=\mathbb{E}\bigl[\mathbb{E}\bigl[\hat{\chi}_{i}\given\xi_{1},\ldots,\xi_{i-1}\bigr]\bigr]=\mathbb{E}\left[\frac{k_{0}s_{0}+\sum_{j=1}^{i-1}\xi_{j}s_{0}}{1+\sum_{j=1}^{i-1}\mu_{j}s_{0}}\right].
\end{equation}

As $i\to\infty$, the first term in \eqref{eq:explaw} converges to zero and, since $\mathbb{E}[\xi_{i}]=\mu_{i}\overline{\chi}$, the second term converges to $\overline{\chi}$.
\end{proof}

Note that this convergence result holds regardless of the values chosen for parameters $k_{0}$ and $s_{0}$ that specify the distribution of the prior knowledge on the external factor.

\subsection{Optimal Restocking}\label{sec:bayoptrest}
The updated knowledge about the external factor after customer $i-1$ is visited, represented by r.v. $\hat{\chi}_{i}$, allows us to update the prior on the demand rate of customer $i$, represented by r.v. $\lambda_{i}=\mu_{i}\hat{\chi}_{i}$. Considering again the Poisson-gamma model \citep{hoff2009first}, the posterior predictive of $\xi_{i}$ given $x^{(1)},\ldots,x^{(i-1)}$ follows a negative binomial distribution with parameters $k_{0}+X_{i-1}$ (maximum number of failures in the experiment) and $\rho_{i}/(1+\rho_{i})$ (success probability of each trial), where $\rho_{i}$ is given by:
\begin{equation}\nonumber
\rho_{i}=\frac{s_{0}\mu_{i}}{1+s_{0}\sum_{j=1}^{i-1}\mu_{j}}.
\end{equation}

Hence, the posterior predictive demand distribution of customer $i$ depends on the \emph{sum} of actual demands observed until customer $i-1$ (that is, on $X_{i-1}$), and not individually on observations $x^{(1)},\ldots,x^{(i-1)}$. This suggests a tractable extension of the optimal restocking algorithm from Section~\ref{sec:optrestock} to the case of correlated demands, in which a new state variable is introduced in the dynamic program to keep track of $X_{i}$.

We now present the optimal restocking algorithm for the case of correlated demands. Rather than defining new symbols, we recycle the notation already introduced in Section \ref{sec:optrestock} by ``overloading'' functions with the new state variable. The auxiliary functionals are redefined as follows:
\begin{gather*}
\phi'(i,j,q,X,\Phi(\cdot,\cdot))=c_{ij}+\mathbb{E}\Bigl[\Upsilon_{\xi_{j}}^{q}(c_{j0}+c_{0j})+\Phi(q+Q\Upsilon_{\xi_{j}}^{q}-\xi_{j},X+\xi_{j})\Bigr],\\
\phi''(i,j,X,\Phi(\cdot,\cdot))=c_{i0}+c_{0j}+\mathbb{E}\Bigl[\Upsilon_{\xi_{j}}^{Q}(c_{j0}+c_{0j})+\Phi(Q+Q\Upsilon_{\xi_{j}}^{Q}-\xi_{j},X+\xi_{j})\Bigr],
\end{gather*}
where $\Phi:\{0,\ldots,Q\}\times\mathbb{N}\mapsto\mathbb{R}$ is the cost-to-go once customer $j$ is fully served, as a function of the remaining capacity \emph{and} of the total demand already observed.

Given a route $\theta=(v_{1},\ldots,v_{H})$, we denote by $\Phi_{\theta}^{k}(q,X)$ the expected cost-to-go along $\theta$ once customer $v_{k}$ is fully served, the remaining vehicle capacity is $q$, and the total demand of customers $v_{1},\ldots,v_{k}$ is $X$. We compute this cost by solving the dynamic program:
\begin{equation}\nonumber
\Phi_{\theta}^{k}(q,X)=\begin{cases}
\min\bigl\{\phi'(v_{k},v_{k+1},q,X,\Phi_{\theta}^{k+1}(\cdot,\cdot)),\phi''(v_{k},v_{k+1},X,\Phi_{\theta}^{k+1}(\cdot,\cdot))\bigr\},& \text{if }k<H,\\
c_{v_{H}0},& \text{if }k=H.
\end{cases}
\end{equation}

Finally, the cost of route $\theta$ under optimal restocking, assuming correlated demands, is given by:
\begin{equation}\nonumber
g^{*}(\theta)=\phi'(0,v_{1},Q,0,\Phi_{\theta}^{1}(\cdot,\cdot)).
\end{equation}

Under the assumption that $\mathbb{P}[\xi_{i}\leq Q]=1$ for all $i\in\theta$, the dynamic programming algorithm has a time complexity of $\mathcal{O}(|\theta|^{2}Q^{3})$, which compares with a complexity of $\mathcal{O}(|\theta|Q^{2})$ for the case of independent demands.

\section{Computational Results}\label{sec:results}
We perform two sets of computational experiments. In Section \ref{sec:resbpnc}, we compare the elementary BP\&C algorithm (E-BP\&C) from Section \ref{sec:bpnc} with the dominance-based BP\&C (D-BP\&C) by \cite{Florio_2020}. In Section \ref{res:corr}, we assess the benefit of applying the restocking policy proposed in Section \ref{sec:bayoptrest} when customer demands are positively correlated.

The E-BP\&C algorithm is implemented in C++ and is available at \url{https://github.com/amflorio/vrpsd-optimal-restocking}. All computational experiments are performed on a single thread of an Intel\textsuperscript{\textregistered} Xeon\textsuperscript{\textregistered} E5-2683 v4 (2.10GHz) processor with 64GB of available memory, and IBM\textsuperscript{\textregistered} CPLEX\textsuperscript{\textregistered} version 12.10 is used to solve linear programs.

\subsection{Elementary Branch-Price-and-Cut}\label{sec:resbpnc}
We base our comparison on instances of the sets A, E and P of the CVRPLIB \citep{uchoa2017new}, and also instances of the set X with a low $n/m$ ratio. In total, we consider $45$ instances with 36 to 194 customers and $n/m$ ratios varying from $3$ to $10$. \hl{Customer demands follow Poisson distributions with rates as given by the deterministic demands in the original instances. Similarly to }\cite{Florio_2020}\hl{, we truncate the Poisson distributions at $\varepsilon=10^{-5}$, meaning that all probabilities less than $\varepsilon$ are set to zero. Still, the algorithm remains effective for smaller $\varepsilon$ values such as $10^{-6}$ or $10^{-7}$.}

As determined empirically by \cite{Florio_2020}, many of the considered instances exhibit a significant VSS (i.e., 5\% or more). We note that neither D-BP\&C nor E-BP\&C is competitive for solving the instances proposed by \cite{Louveaux_2018}, which have large $n/m$ ratios (up to 50) but display, on average, very small VSS (i.e., less than 1\%). \hl{Best overall results are obtained by employing only RCSP completion bounds, as the computation of knapsack bounds require additional runtime that, in most instances, is larger than the runtime savings observed in the pricing algorithm. Therefore, except where noted otherwise, the results in this section are produced by a BP\&C implementation that employs RCSP bounds only.}

\begin{table}[t]
\centering
\begingroup
\small
\renewcommand{\arraystretch}{0.7750}
\caption{\label{tab:ebpncfreeA}Results E-BP\&C: Unlimited Fleet Size, $f=1.00$ (Set A)}
\begin{tabular}{lrrrrrrrrrr}
\toprule
& \multicolumn{4}{c}{D-BP\&C} & \multicolumn{6}{c}{E-BP\&C} \\
\cmidrule(l){2-5}\cmidrule(l){6-11}
Instance & Best & Paths & Gap & Time & Best & Paths & Gap & Time & RCCs & SRCs \\
\midrule
A-n36-k5 & 866.740 & 5 & 1.21\% & 300 & 862.309 & 5 &  & 34 & 29 & 80 \\
A-n37-k5 & 710.068 & 5 &  & 47 & 710.068 & 5 &  & 3 & 17 & 40 \\
A-n37-k6 & 1033.08 & 7 &  & 7 & 1033.08 & 7 &  & 1 & 39 & 40 \\
A-n38-k5 & 779.406 & 6 &  & 17 & 779.406 & 6 &  & 1 & 16 & 32 \\
A-n39-k5 & 875.618 & 6 &  & 2 & 875.618 & 6 &  & 1 & 17 & 0 \\
A-n39-k6 & 877.922 & 6 &  & 6 & 877.922 & 6 &  & 1 & 20 & 24 \\
A-n44-k6 & 1026.45 & 7 &  & 155 & 1026.45 & 7 &  & 16 & 18 & 82 \\
A-n45-k6 & 1029.68 & 7 &  & 78 & 1029.68 & 7 &  & 5 & 40 & 74 \\
A-n45-k7 & 1264.94 & 7 &  & 14 & 1264.94 & 7 &  & 2 & 13 & 16 \\
A-n46-k7 & 1003.23 & 7 &  & 14 & 1003.23 & 7 &  & 1 & 27 & 16 \\
A-n48-k7 & 1189.10 & 7 & 0.07\% & 300 & 1189.10 & 7 &  & 17 & 40 & 16 \\
A-n53-k7 & 1127.19 & 8 & 0.43\% & 300 & $^*$1127.22 & 8 & 0.22\% & 600 & 26 & 64 \\
A-n54-k7 & 1292.53 & 8 & 0.14\% & 300 & 1292.53 & 8 &  & 33 & 37 & 64 \\
A-n55-k9 & 1183.93 & 10 &  & 9 & 1183.93 & 10 &  & 1 & 34 & 24 \\
A-n60-k9 & 1535.52 & 10 & 1.65\% & 300 & 1533.38 & 10 & 0.57\% & 600 & 88 & 144 \\
A-n61-k9 & 1157.86 & 10 & 1.54\% & 300 & 1149.00 & 10 &  & 133 & 60 & 152 \\
A-n62-k8 & -- & -- & -- & 300 & 1434.71 & 9 &  & 185 & 30 & 16 \\
A-n63-k9 & -- & -- & -- & 300 & 1850.81 & 10 & 0.14\% & 600 & 51 & 64 \\
A-n63-k10 & 1463.24 & 11 & 0.62\% & 300 & 1464.06 & 11 & 0.45\% & 600 & 36 & 48 \\
A-n64-k9 & -- & -- & -- & 300 & 1583.26 & 10 & -- & 600 & 20 & 0 \\
A-n65-k9 & 1318.95 & 10 &  & 121 & 1318.95 & 10 &  & 4 & 41 & 64 \\
A-n69-k9 & 1264.42 & 10 & 0.31\% & 300 & $^*$1264.42 & 10 & 0.17\% & 600 & 32 & 32 \\
\bottomrule
\end{tabular}
\caption*{$^*$Results obtained with knapsack bounds enabled.}
\endgroup
\end{table}

\begin{table}[t]
\centering
\begingroup
\small
\renewcommand{\arraystretch}{0.7750}
\caption{\label{tab:ebpncfreeEPX}Results E-BP\&C: Unlimited Fleet Size, $f=1.00$ (Sets E, P and X)}
\begin{tabular}{lrrrrrrrrrr}
\toprule
& \multicolumn{4}{c}{D-BP\&C} & \multicolumn{6}{c}{E-BP\&C} \\
\cmidrule(l){2-5}\cmidrule(l){6-11}
Instance & Best & Paths & Gap & Time & Best & Paths & Gap & Time & RCCs & SRCs \\
\midrule
E-n51-k5 & 556.800 & 6 & 0.83\% & 300 & $^*$556.800 & 6 & 0.37\% & 600 & 39 & 160 \\
E-n76-k7 & 703.811 & 7 & 0.42\% & 300 & 705.560 & 7 & 0.45\% & 600 & 33 & 48 \\
E-n76-k8 & 778.423 & 9 & 1.04\% & 300 & 777.721 & 9 & 0.30\% & 600 & 43 & 128 \\
E-n76-k10 & 891.224 & 11 &  & 227 & 891.224 & 11 &  & 8 & 31 & 64 \\
E-n76-k14 & 1122.36 & 16 &  & 228 & 1122.36 & 16 &  & 16 & 60 & 80 \\
E-n101-k14 & 1182.99 & 15 & 0.69\% & 300 & 1182.99 & 15 & 0.32\% & 600 & 117 & 382 \\
P-n40-k5 & 475.705 & 5 &  & 0 & 475.705 & 5 &  & 0 & 12 & 0 \\
P-n45-k5 & 537.237 & 5 & 0.15\% & 300 & $^*$537.237 & 5 &  & 191 & 19 & 136 \\
P-n50-k7 & 587.314 & 7 &  & 5 & 587.314 & 7 &  & 3 & 24 & 80 \\
P-n50-k8 & 673.430 & 9 &  & 6 & 673.430 & 9 &  & 1 & 34 & 72 \\
P-n50-k10 & 763.212 & 11 &  & 17 & 763.212 & 11 &  & 2 & 45 & 64 \\
P-n51-k10 & 814.687 & 11 &  & 1 & 814.687 & 11 &  & 0 & 21 & 24 \\
P-n55-k7 & 591.846 & 7 &  & 18 & 591.846 & 7 &  & 2 & 19 & 72 \\
P-n55-k10 & 746.080 & 10 &  & 25 & 746.080 & 10 &  & 1 & 23 & 80 \\
P-n55-k15 & 1073.11 & 18 &  & 1 & 1073.11 & 18 &  & 1 & 40 & 11 \\
P-n60-k10 & 805.176 & 11 &  & 74 & 805.176 & 11 &  & 5 & 56 & 80 \\
P-n60-k15 & 1089.23 & 16 &  & 0 & 1089.23 & 16 &  & 1 & 36 & 16 \\
P-n65-k10 & 857.481 & 10 & 0.29\% & 300 & 857.481 & 10 &  & 229 & 125 & 213 \\
P-n70-k10 & 888.060 & 11 & 0.02\% & 300 & 888.060 & 11 &  & 8 & 53 & 72 \\
X-n101-k25 & 30606.9 & 29 & 0.59\% & 300 & 30606.9 & 29 & 0.06\% & 600 & 279 & 113 \\
X-n110-k13 & 16734.1 & 14 & 0.33\% & 300 & 16734.0 & 14 &  & 182 & 64 & 80 \\
X-n148-k46 & 56584.4 & 50 & 0.01\% & 300 & 56584.4 & 50 &  & 111 & 341 & 0 \\
X-n195-k51 & 50025.3 & 58 & 0.57\% & 300 & 49955.5 & 58 & 0.17\% & 600 & 151 & 32 \\
\bottomrule
\end{tabular}
\caption*{$^*$Results obtained with knapsack bounds enabled.}
\endgroup
\end{table}

Tables \ref{tab:ebpncfreeA} and \ref{tab:ebpncfreeEPX} compare the results obtained by D-BP\&C \citep[extracted from][]{Florio_2020} and E-BP\&C on all 45 instances when assuming an unlimited vehicle fleet and a load factor $f=1.00$. Column ``Time'' reports the runtime in minutes; columns RCCs and SRCs report the number of RCCs and SRCs separated by the algorithm. A dash (`--') in column ``Best'' indicates that no upper bound could be found, and a dash in column ``Gap'' indicates that no lower bound could be found. The results show that, with the exception of a single instance (E-n76-k7), E-BP\&C completely dominates D-BP\&C in terms of number of instances solved, solution time and optimality gap. We remark that the processor used in our experiments is ranked approximately 20\% faster than the processor employed by \cite{Florio_2020}, but the runtime savings by E-BP\&C exceed that amount by a large margin.

\begin{table}[thpb]
\centering
\begingroup
\small
\renewcommand{\arraystretch}{0.7750}
\caption{\label{tab:ebpncfix}Results E-BP\&C: Fixed Fleet Size, $f=1.00$}
\begin{tabular}{lrrrrrrrr}
\toprule
& \multicolumn{3}{c}{D-BP\&C} & \multicolumn{5}{c}{E-BP\&C} \\
\cmidrule(l){2-4}\cmidrule(l){5-9}
Instance & Best & Gap & Time & Best & Gap & Time & RCCs & SRCs \\
\midrule
A-n37-k6 & 1044.08 &  & 5 & 1044.08 &  & 2 & 28 & 8 \\
A-n38-k5 & 808.874 &  & 62 & 808.874 &  & 28 & 26 & 96 \\
A-n39-k5 & 887.547 &  & 51 & 887.547 &  & 7 & 18 & 16 \\
A-n44-k6 & 1029.60 &  & 7 & 1029.60 &  & 4 & 26 & 0 \\
A-n45-k6 & 1090.14 & 0.34\% & 300 & 1090.14 &  & 7 & 59 & 24 \\
A-n53-k7 & -- & -- & 300 & 1151.64 &  & 375 & 64 & 64 \\
A-n54-k7 & 1301.58 & 0.29\% & 300 & 1301.58 &  & 31 & 47 & 16 \\
A-n55-k9 & 1202.70 &  & 123 & 1202.70 &  & 8 & 93 & 48 \\
A-n60-k9 & 1556.67 & 2.36\% & 300 & 1531.37 &  & 332 & 50 & 88 \\
A-n61-k9 & 1190.23 & 1.33\% & 300 & 1190.23 &  & 86 & 108 & 80 \\
A-n62-k8 & -- & -- & 300 & 1454.03 & 1.50\% & 600 & 43 & 8 \\
A-n63-k9 & -- & -- & 300 & 1920.49 & -- & 600 & 48 & 0 \\
A-n63-k10 & -- & -- & 300 & 1479.97 &  & 361 & 68 & 96 \\
A-n64-k9 & -- & -- & 300 & 1597.47 & -- & 600 & 43 & 0 \\
A-n65-k9 & 1356.85 & 0.47\% & 300 & 1356.85 &  & 32 & 104 & 24 \\
A-n69-k9 & 1282.02 & 1.32\% & 300 & $^*$1277.68 & 0.58\% & 600 & 57 & 16 \\
E-n51-k5 & 556.795 &  & 183 & 556.795 &  & 40 & 42 & 24 \\
E-n76-k7 & 746.534 & 7.06\% & 300 & 711.797 & 1.58\% & 600 & 27 & 16 \\
E-n76-k8 & 804.390 & 4.49\% & 300 & $^*$779.926 & 0.90\% & 600 & 36 & 40 \\
E-n76-k10 & 918.224 & 1.29\% & 300 & 912.839 &  & 502 & 165 & 120 \\
E-n76-k14 & 1160.19 & 0.44\% & 300 & 1159.37 &  & 264 & 239 & 88 \\
E-n101-k14 & 1187.79 & 0.97\% & 300 & 1182.57 &  & 322 & 94 & 56 \\
P-n50-k8 & 713.671 & 1.12\% & 300 & 713.188 &  & 166 & 129 & 144 \\
P-n50-k10 & 775.579 &  & 8 & 775.579 &  & 2 & 72 & 48 \\
P-n51-k10 & 850.742 & 0.18\% & 300 & 850.742 &  & 6 & 130 & 56 \\
P-n55-k15 & 1176.60 & 0.10\% & 300 & 1176.60 &  & 19 & 261 & 80 \\
P-n60-k10 & 814.436 &  & 34 & 814.436 &  & 3 & 60 & 24 \\
P-n60-k15 & 1109.88 &  & 37 & 1109.88 &  & 4 & 135 & 40 \\
P-n70-k10 & 912.974 & 2.52\% & 300 & 912.974 &  & 469 & 184 & 144 \\
X-n101-k25 & 37273.2 & 3.06\% & 300 & 36894.7 & 0.96\% & 600 & 204 & 40 \\
X-n110-k13 & 17201.1 & 2.16\% & 300 & $^*$17072.7 & 1.26\% & 600 & 102 & 32 \\
X-n148-k46 & 57223.5 & 0.13\% & 300 & 57207.1 & 0.02\% & 600 & 1032 & 76 \\
X-n195-k51 & -- & -- & 300 & 58382.4 & 0.93\% & 600 & 377 & 16 \\
\bottomrule
\end{tabular}
\caption*{$^*$Results obtained with knapsack bounds enabled.}
\endgroup
\end{table}

The good performance of E-BP\&C is attributed to two main reasons. First, E-BP\&C controls more efficiently the combinatorial growth of labels in the pricing algorithm. The D-BP\&C method spends, on average, one-third of the computation time verifying dominance rules. Besides saving this computational effort, E-BP\&C employs strong completion bounds that are more effective than dominance rules in pruning unpromising partial paths in the labeling procedure. Second, E-BP\&C achieves tighter lower bounds by pricing only elementary routes and by separating a large number of SRCs. Such an aggressive separation of SRCs is not possible in D-BP\&C because it hinders label dominance and makes the pricing problem intractable.

Table \ref{tab:ebpncfix} reports the results on the $33$ instances that, when considering an unlimited fleet, either could not be solved to optimality or whose optimal solution uses more vehicles than the minimum required for feasibility. This time, we solve these instances by enforcing the fleet size limitation. The comparison is again very favorable to E-BP\&C. In summary, considering both cases (unlimited and fixed fleet sizes) and a load factor $f=1.00$, E-BP\&C solves the 32 instances solved by D-BP\&C and, additionally, 24 previously unsolved instances. \hl{Among the 24 newly solved instances, 18 are solved within the same 5 hours runtime allowed for D-BP\&C, and 6 are solved within the additional runtime allowed for E-BP\&C.}

We now consider instances with a load factor above $1.00$. Initially, we select the instances from sets A and P that could be solved for the limited fleet size case with load factor $f=1.00$, and resolve these instances with $f=1.25$ and $f=1.50$. We limit the number of vehicles to the minimum required and, to keep roughly the same $n/m$ ratio, we decrease the vehicle capacity to the integer value nearest to $Q_{0}/f$, where $Q_{0}$ is the original vehicle capacity in each instance. The results are reported in Table \ref{tab:lf}.

\begin{table}[t]
\centering
\begingroup
\small
\renewcommand{\arraystretch}{0.7750}
\caption{\label{tab:lf}Results E-BP\&C: Fixed Fleet Size, $f\in\{1.25,1.50\}$}
\begin{tabular}{lrrrrrr}
\toprule
 & \multicolumn{3}{c}{$f=1.25$} & \multicolumn{3}{c}{$f=1.50$} \\
 \cmidrule(l){2-4}\cmidrule(l){5-7}
Instance & Best & Gap & Time & Best & Gap & Time \\
\midrule
A-n36-k5 & 1025.51 &  & 265 & 1169.90 & 2.35\% & 600 \\
A-n37-k6 & 1274.22 &  & 265 & 1422.53 &  & 258 \\
A-n38-k5 & 956.391 &  & 259 & 1027.86 &  & 210 \\
A-n39-k5 & 1067.11 &  & 543 & 1204.04 & -- & 600 \\
A-n39-k6 & 1035.32 &  & 35 & 1152.73 &  & 134 \\
A-n44-k6 & 1278.70 & 1.09\% & 600 & 1390.14 & 0.92\% & 600 \\
A-n45-k6 & 1296.62 &  & 76 & 1380.03 &  & 139 \\
A-n45-k7 & 1618.09 & 1.12\% & 600 & 1856.26 & -- & 600 \\
A-n48-k7 & 1500.52 & -- & 600 & 1714.14 & -- & 600 \\
A-n53-k7 & 1448.50 & -- & 600 & 1603.92 & -- & 600 \\
A-n54-k7 & 1674.25 & -- & 600 & 1901.88 & -- & 600 \\
A-n55-k9 & 1521.48 &  & 359 & 1675.36 & 0.44\% & 600 \\
A-n60-k9 & 1982.64 & -- & 600 & 2206.92 & -- & 600 \\
A-n61-k9 & 1457.36 & 0.07\% & 600 & 1589.73 & 0.44\% & 600 \\
A-n63-k10 & 1849.51 & -- & 600 & 2090.84 & -- & 600 \\
A-n65-k9 & 1761.26 & -- & 600 & 1918.98 & -- & 600 \\
P-n40-k5 & 560.990 &  & 87 & 611.133 &  & 166 \\
P-n45-k5 & 628.288 &  & 181 & 692.645 & 2.62\% & 600 \\
P-n50-k7 & 708.412 &  & 120 & 779.710 & 0.16\% & 600 \\
P-n50-k8 & 847.160 &  & 68 & 913.388 &  & 192 \\
P-n50-k10 & 948.558 &  & 45 & 1027.28 &  & 5 \\
P-n51-k10 & 1034.66 &  & 49 & 1126.52 &  & 64 \\
P-n55-k10 & 915.545 &  & 20 & 999.292 &  & 10 \\
P-n55-k15 & 1391.91 &  & 2 & 1497.86 &  & 29 \\
P-n60-k10 & 1004.45 &  & 255 & 1101.04 & 0.02\% & 600 \\
P-n60-k15 & 1372.09 &  & 2 & 1526.30 &  & 26 \\
P-n65-k10 & 1055.37 & 0.004\% & 600 & 1153.57 & 0.13\% & 600 \\
P-n70-k10 & 1106.96 & 0.07\% & 600 & 1216.71 & 0.86\% & 600 \\
\bottomrule
\end{tabular}
\endgroup
\end{table}

Out of the 28 instances, 17 can be solved when $f=1.25$, and 11 can be solved when $f=1.50$. A comparison with D-BP\&C does not apply in this case, since D-BP\&C relies on heuristic dominance rules to solve instances with $f>1$. In fact, no previous work reports optimal solutions to VRPSD instances with $f>1$. Even though ratios $n/m$ remain constant, problem difficulty increases markedly with larger load factors, mainly because the completion bounds lose effectiveness. In particular, the RCSP bound does not consider restocking costs in the forward path, and these costs become more relevant when vehicles must replenish more often. As a result, the cost of the forward path is underestimated and the resulting bound becomes weaker.

\subsection{Optimal Restocking for Correlated Demands}\label{res:corr}
\hl{Customer demands are correlated whenever they depend on the same covariates. For example, the weather affects the demand for soft drinks, and city events affect the amount of waste to be collected from public thrash cans. Hence, demand correlation plays an important role in applications such as replenishment of vending machines and waste collection.} Our next goal is to measure the cost savings when adopting a restocking policy that takes account of positive demand correlation. To this end, we compare the policy proposed in Section \ref{sec:bayoptrest} (OR-C) with the classical optimal restocking policy (OR-I), in which demands are assumed independent. The comparison is performed by simulating both policies over a large number of routes and under correlated demands. The savings may be interpreted as the \emph{cost of independence}, that is, the cost of assuming independent demands when such assumption does not hold.

In these experiments, we consider the routes from the simulation study in \cite{FlorioReopt2021}. For each $n\in\{3,\ldots,15\}$, 20 routes are generated as follows. First, $n$ customers are randomly placed on a 1,000 by 1,000 grid, and the depot is located at the southwest corner of the grid. Then, a route $\theta$ is defined by solving a traveling salesman problem on the set of nodes. The demand of each customer $i\in\theta$ is Poisson distributed with rate $\mu_{i}\tilde{\chi}$, where $\mu_{i}$ and $\tilde{\chi}$ are uniformly distributed on the intervals $[10,100]$ and $[0.5, 1.5]$, respectively. We consider load factors $f\in\{1.3,1.6,1.9,2.5\}$ by setting the vehicle capacity to the integer value closest to $\sum_{i\in\theta}\mu_{i}/f$. For each route and load factor, we simulate and calculate the costs of both policies under $5,000$ randomly generated demand scenarios.

Table \ref{tab:orc} reports, for each route length (in terms of number of customers) and load factor, the average and maximum savings realized by OR-C over all simulations. As we see, the cost savings by OR-C may be quite substantial (in excess of 10\%). The savings are significant across all route lengths and load factors, which highlights the importance of taking demand correlation into account when prescribing replenishment decisions.

\begin{table}[t!]
\centering
\begingroup
\small
\renewcommand{\arraystretch}{0.7750}
\caption{\label{tab:orc}Cost Savings when Taking Account of Correlation}
\begin{tabular}{rrrrrrrrrr}
\toprule
 & \multicolumn{2}{c}{$f=1.3$} & \multicolumn{2}{c}{$f=1.6$} & \multicolumn{2}{c}{$f=1.9$} & \multicolumn{2}{c}{$f=2.5$} \\
\cmidrule(l){2-3}\cmidrule(l){4-5}\cmidrule(l){6-7}\cmidrule(l){8-9}
$n$ & Avg(\%) & Max(\%) & Avg(\%) & Max(\%) & Avg(\%) & Max(\%) & Avg(\%) & Max(\%) & Avg \\
\midrule
3 & 2.35 & 9.05 & 3.66 & 9.51 & 2.01 & 6.49 & 1.36 & 7.31 & 2.35 \\
4 & 1.67 & 7.31 & 3.06 & 7.96 & 1.74 & 4.29 & 2.47 & 6.92 & 2.24 \\
5 & 2.08 & 8.07 & 2.94 & 8.43 & 2.22 & 5.41 & 1.96 & 5.81 & 2.30 \\
6 & 1.30 & 4.05 & 3.20 & 9.81 & 2.68 & 8.98 & 2.01 & 6.44 & 2.30 \\
7 & 2.11 & 7.57 & 4.00 & 9.86 & 2.20 & 6.65 & 1.63 & 3.77 & 2.48 \\
8 & 1.41 & 4.84 & 2.72 & 7.30 & 1.81 & 4.48 & 2.17 & 8.93 & 2.03 \\
9 & 1.77 & 6.42 & 2.15 & 5.42 & 2.63 & 7.33 & 2.31 & 5.03 & 2.22 \\
10 & 1.41 & 3.52 & 2.94 & 7.36 & 2.93 & 7.55 & 2.33 & 6.99 & 2.40 \\
11 & 1.28 & 5.42 & 2.63 & 8.15 & 2.69 & 7.03 & 2.22 & 7.34 & 2.20 \\
12 & 1.42 & 6.39 & 2.63 & 10.19 & 3.10 & 8.13 & 2.32 & 6.36 & 2.37 \\
13 & 0.88 & 3.70 & 2.65 & 10.86 & 2.61 & 8.89 & 2.18 & 5.66 & 2.08 \\
14 & 2.41 & 7.48 & 2.37 & 6.29 & 2.17 & 5.93 & 3.08 & 7.86 & 2.51 \\
15 & 1.56 & 4.93 & 2.33 & 6.83 & 2.46 & 9.55 & 2.32 & 4.86 & 2.17 \\
Avg & 1.67 & & 2.87 & & 2.40 & & 2.18 & & 2.28 \\
\bottomrule
\end{tabular}
\endgroup
\end{table}

\section{Conclusions}\label{sec:conclusions}
\subsection{Main Contributions and Takeaways}
Recent years have witnessed substantial progress on VRPSD research. We reviewed the new exact methods for solving the problem, under both the restocking-based and the chance-constraint perspectives, and discussed the main modeling and algorithmic challenges related to each perspective. Concerning methodology, our original contributions are twofold. First, we introduced a new state-of-the-art algorithm for the VRPSD under optimal restocking. The E-BP\&C algorithm outperforms considerably the dominance-based BP\&C by \cite{Florio_2020} and allows the solution of several previously open literature instances. Second, we proposed a demand model for handling positively correlated demands and a restocking policy that takes account of demand correlation when prescribing preventive replenishment decisions.

The recent literature on the VRPSD covered some of the research gaps identified by \cite{Gendreauetal2016}, namely, how to address demand correlation and how to solve the restocking-based VRPSD under more sophisticated restocking policies. Demand correlation is now addressed in a few works, which allows us assess the cost of independence, that is, the cost of falsely assuming independent customer demands. This cost is quite significant when restocking is allowed, as measured in Section \ref{res:corr}: savings may exceed 10\% when correlation is taken into account. In the chance-constrained VRPSD, the cost of independence is a feasibility cost, in the sense that a solution obtained by assuming independent demands might not be feasible under correlated demands. Back to the restocking perspective, we notice a trend in solving the VRPSD under increasingly complex recourse actions. Until the review by \cite{Gendreauetal2016}, all algorithms for the restocking-based VRPSD assumed the detour-to-depot policy. Since then, several exact methods based on preventive restocking have been introduced. All in all, the recent advances helped to bridge the gap between theory and practice in vehicle routing with uncertain demands.

\subsection{The Way Forward}
We conclude this paper with a few research perspectives:

\textbf{Inter-route recourse policies.}
The results by \cite{FlorioReopt2021} show that there is only marginal benefit in allowing, during route execution, partial reoptimization of the customer visiting sequence. At the same time, there is currently no exact algorithm to solve the VRPSD under a \emph{pooling} recourse policy, in which vehicles may cooperate to serve all demand. \hl{Potential applications include package collection, in which drivers could deviate from their planned routes in order to support other drivers who are facing unexpectedly high demand.} Although cooperative recourse strategies have already been explored \citep[see, e.g.,][]{ak2007,zhu2014}, there is currently no measure of the benefits of such class of policies compared to optimal restocking. Assessing this benefit is crucial to determine whether the additional complexity of coordinating a vehicle fleet in real-time is justified.

\textbf{Distribution fairness.}
The chance-constrained VRPSD assumes that all unserved demand is ignored, at zero penalty. In a practical setting, the same VRPSD solution may be implemented at several periods, until there is enough new data to update the customer demand distributions and to solve the problem again with the updated distributions. Hence, a fairness issue arises as the same customers (i.e., the last ones along the routes) may repeatedly be underserved each time the solution is implemented and the vehicle capacity is exceeded. We believe that one promising avenue to promote a more balanced service is to implement chance-constrained solutions under a \emph{distribution fairness} policy. In this type of policy, each time a customer is visited a decision is made concerning the demand amount that is actually served, with the double goal of serving demands as much as possible and in a balanced way. \hl{An important problem in humanitarian logistics concerns the distribution of supplies to medical facilities to meet emergency (and uncertain) demand}\citep{Liu_2022}\hl{. Such problem setting would benefit from an allocation policy to distribute resources fairly among all customers.}

\textbf{Duration constraints.}
In many practical settings of restocking-based VRPSDs, ensuring that drivers return to the depot within a given timespan (e.g., due to working hours regulations) is more important than limiting (in expectation) the total demand along a route. Conversely, the number of replenishment trips actually performed may not be as important as finishing the operation on time. \hl{For example, in garbage collection applications trucks may typically unload at a landfill a few times per day.} Nevertheless, the method by \cite{Florio_2021}, which is only effective for solving instances with tight duration limits, remains the only exact algorithm for the VRPSD with duration constraints. Enforcing only duration constraints is significantly more challenging, because route duration is, itself, random, and because bounds derived from limiting the total demand are no longer available. Hence, extensive research is still required before more practical instances of this important problem variant can be solved.

\textbf{Demand correlation.}
Even though recent work explored some facets of demand correlation, there is still considerable ground to cover towards a more general methodology to deal with VRPSDs under correlation. \hl{As mentioned in Section }\ref{res:corr}\hl{, correlation appears whenever stochastic demands depend on a common factor, which may indeed be the case in applications such as heating oil delivery and waste collection.} The Bayesian demand model and restocking policy proposed in Section~\ref{sec:corr} treat the case of positive demand correlation among all customers. It is still an open question, however, how to prescribe optimal replenishment decisions given a fixed sequence of customers under an arbitrary demand correlation structure. Likewise, there is currently no exact method that is able to solve the restocking-based VRPSD under correlation.

\section*{Acknowledgement}
The authors are thankful to the editors and a reviewer for insightful comments and suggestions that helped us to improve this manuscript.

%\bibliography{mybibfile}

\end{document}